\documentclass[preprint,12pt]{elsarticle}

\usepackage{amssymb}
\usepackage{amsmath}
\usepackage{amsfonts}
\usepackage{graphicx}
\usepackage{graphics}
\usepackage{tikz}
\usepackage{titlesec}

\newtheorem{theorem}{Theorem}

\newtheorem{corollary}[theorem]{Corollary}

\newtheorem{definition}[theorem]{Definition}
\newtheorem{example}[theorem]{Example}

\newtheorem{lemma}[theorem]{Lemma}

\newtheorem{problem}[theorem]{Problem}
\newtheorem{proposition}[theorem]{Proposition}
\newtheorem{remark}[theorem]{Remark}

\def\qed{\vbox{\hrule
 \hbox{\vrule\hbox to 5pt{\vbox to 8pt{\vfil}\hfil}\vrule}\hrule}}

\usepackage{tikz}
\usepackage{tkz-graph}

\journal{xxxxxxx}

\begin{document}

\begin{frontmatter}


\title{Guo's index for some classes of matrices}

\author{Mar\'{i}a Robbiano}

\address{Departamento de Matem\'{a}ticas, Facultad de Ciencias. Universidad Cat\'{o}lica del Norte. Av. Angamos 0610 Antofagasta, Chile.}
\ead{mrobbiano@ucn.cl}

\begin{abstract}

\noindent A permutative matrix is a square matrix such that every row is a permutation of the first row. A circulant matrix is a matrix where each row is a cyclic shift of the row above to the right.
The Guo's index $\lambda_0$ of a realizable list is the minimum spectral radius such that the list (up to the initial spectral radius) together with $\lambda_0$
is realizable.
The Guo's index of some permutative matrices is obtained. Our results are constructive. Some examples designed using MATLAB are given at the end of the paper.
\end{abstract}


\end{frontmatter}
\section{Preliminaries}
\noindent In this section we present a brief resume.  Recall that a square matrix $A=(a_{ij})$ is nonnegative ($A \geq 0$) if and only if $a_{ij} \geq 0$ \  $ \left( 1\leq i,j\leq n\right).$ 
A list is an $n$-tuple, $\Lambda =\left(\lambda_1,\ldots,\lambda_n\right)$ of complex numbers and it is \textit{realized} by a $n$-by $n$ nonnegative matrix $A$ if the set formed its components and the spectrum of $A$ (considering multiplicities) coincide.
The NIEP is the problem of determining necessary and sufficient conditions for a list of $n$ complex
numbers to be realized by an $n$-by-$n$ nonnegative matrix $A$. If a list $\Lambda$ is realized by a nonnegative matrix $A$, then $\Lambda$ is \textit{realizable} and the matrix $A$ \textit{realizes} $\Lambda$ (or, that is a  \textit{realizing matrix} for the list). Some results can be seen in e.g. \cite{Laffey1,Laffey2,Laffey-Smigoc}.
 A variant of the original problem is the question for which lists of $n$ real numbers can occur as eigenvalues of an $n$-by-$n$  nonnegative matrix and it is called real nonnegative inverse eigenvalue problem (RNIEP). Some results can be seen in e.g. \cite{Laffey}. The structured NIEP is an analogous problem to NIEP where the realizing matrix must be structured, for instance, the matrix can be symmetric, Toeplitz, Hankel, circulant, normal, permutative, etc., see in \cite{Fiedler, Laffey, LMc, MAR,PP} and the reference therein.
 In this paper we deal with structured matrices. In particular, permutative and circulant matrices.

\noindent Throughout the text, $\sigma\left(A\right)$  denotes the set of eigenvalues of a square matrix $A$. As usual, the identity matrix of order $n$ is denoted by $I_n$ and if the order of the identity matrix can be easily deduced then it is just denoted by $I$.

\noindent Since a nonnegative matrix is real, its characteristic polynomial must have real coefficients and then
$ \{\lambda_{0}, \ldots, \lambda_{n-1}\} =\Lambda= \overline{\Lambda} =  \{\overline{\lambda_{0}}, \ldots, \overline{\lambda_{n-1}}\},$ where $\overline{\lambda}$ stands for the complex conjugate of $\lambda \in \mathbb{C}.$

\noindent Therefore consider the following definition:

\begin{definition}
The complex $n$-tuple $(\lambda_{0},\ldots,\lambda_{n-1})$ is  \textit{closed under complex conjugation} if the condition $$\{\lambda_{0},\ldots,\lambda_{n-1}\}=\{\overline{\lambda_{0}},\ldots,\overline{\lambda_{n-1}}\},$$ holds.
\end{definition}

\noindent The Perron-Frobenius theory of nonnegative matrices \cite{Berman} plays in this problem an important role.
The theory provides several important necessary conditions for the NIEP. See below some of these conditions resumed. Here, for $1\leq k \leq n,$ \ $s_{k} (\Lambda)=\sum\limits_{i=0}^{n-1}\lambda
_{i}^{k}$ is named the $k$-th moment.

\noindent Some necessary conditions for the list  $\Lambda = (\lambda_{0}, \ldots, \lambda_{n-1})$ of complex numbers to be the spectrum of a nonnegative matrix are:
\begin{enumerate}
\item The spectral radius,  $\max \left\{ \left\vert \lambda \right\vert
:\lambda \in \Lambda \right\} $, called the Perron eigenvalue,  belongs to $\Lambda.$

\item The list $\Lambda $ is closed under complex conjugation.

\item $s_{k}\left(\Lambda \right)\geq 0,$ \ $ k \geq 1 $.

\item $s_{k}^{m}\left( \Lambda \right) \leq n^{m-1}s_{km}\left( \Lambda
\right),$ \ $ k \geq 1 $.
\end{enumerate}
 The last condition was
proved by Johnson \cite{Johnson} and independently by Loewy and London \cite{LwyLdn}.

\vspace{0.5 cm}
The following fundamental theorem  was proven in \cite{GUO} and in its statement it is introduced formally the notion of Guo's index.
\begin{theorem}\cite [Theorem 2.1]{GUO}
\label{Guo} Let $(\lambda_{1},\ldots,\lambda_{n-1})$ be a closed under complex conjugation
$(n-1)$-tuple then, there exists a real number $\lambda_0$ (called Guo\ {'}s index) where
$$\max_{{1}\leq j\leq n-1}{|\lambda_j|}\leq\lambda_0.
$$
such that the list $(\lambda,\lambda_{1},\ldots,\lambda_{n-1})$ is realizable by an $n$-by-$n$ nonnegative matrix $A$ if and only if $\lambda\geq \lambda_0$. Furthermore, $\lambda_0\leq 2n\max_{1\leq j\leq n-1}{|\lambda_j|}$.

\end{theorem}

\noindent We define permutative matrix, below.

\begin{definition}
\cite{PP}
A square matrix of order $n$ with $n\geq 2$ is called a \textit{permutative matrix} or permutative when all its rows (up to the first one) are permutations of precisely its first row.
\end{definition}
The spectra of a class of permutative matrices was studied in \cite{MAR}.
In particular, spectral results for matrices partitioned into $2$-by-$2$ symmetric blocks were presented and, using these results sufficient conditions on a given list to be the list of eigenvalues of a nonnegative permutative matrix were obtained and the corresponding permutative matrices were constructed. 

\noindent This paper is structured in $4$ sections. In the second section we obtain necessary and suficient conditions for a list $\Lambda$ to be the spectrum of a class of permutative matrices. In the third section results about circulant and block matrices, given in \cite{AMRH} are revisited.
In the fourth section necessary and sufficient conditions for some lists to be realized by a class of structured matrices are exhibited.

\section{Classes of Permutative Matrices}
\noindent In this section some auxiliary results from \cite{PP} are recalled and some new definitions are introduced.
In \cite{PP} the following results were proven.

\begin{lemma}
\label{pappa copy(1)}\ \cite[Lemma 3.1]{PP} For $\mathbf{x}=\left(
x_{1},x_{2},\ldots,x_{n}\right)  ^{T}\in\mathbb{C}^{n}$, let
\begin{equation}
X=
\begin{pmatrix}
x_{1} & x_{2} & \ldots & x_{i} & \ldots & x_{n-1} & x_{n}\\
x_{2} & x_{1} & \ldots & x_{i} & \ldots & x_{n-1} & x_{n}\\
\vdots & \vdots & \ddots & \vdots & \ddots & \vdots & \vdots\\
x_{i} & x_{2} & \ddots & x_{1} & \ddots & \vdots & \vdots\\
\vdots & \vdots & \vdots & \vdots & \ddots & \vdots & \vdots\\
x_{n-1} & x_{2} & \ldots & \vdots & \vdots & x_{1} & x_{n}\\
x_{n} & x_{2} & \ldots & x_{i} & \ldots & x_{n-1} & x_{1}%
\end{pmatrix}
.\label{mX}%
\end{equation}
Then, the set of eigenvalues of $X$ is given by
\begin{equation}
\sigma(X)=\left\{
{\displaystyle\sum\limits_{i=1}^{n}}
x_{i},x_{1}-x_{2},x_{1}-x_{3},\ldots,x_{1}-x_{n}\right\}.
\end{equation}
\end{lemma}
\begin{definition}
{\rm The list $\Lambda=\left(\lambda_1,\ldots,\lambda_n\right)$ is a \textit{Sule\u{\i}manova spectrum} if the $
\Lambda$ is real list, $\lambda _{1}>0\geq \lambda _{2}\geq
\cdots \geq \lambda _{n}$ and $\sum_{i=1}^{n}\lambda_i \geq 0$.}
\end{definition}
\begin{theorem} \label{pappa} \cite{PP} Let $\Lambda=\left( \lambda_{1},\ldots,\lambda
_{n}\right)  $ be a Sule\u{\i}manova spectrum and consider the $n$-tuple
$\mathbf{x=}\left(  x_{1},x_{2},\ldots,x_{n}\right)  $, where
\[%
\begin{tabular}
[c]{ccc}%
$x_{1}=\frac{\lambda_{1}+\cdots+\lambda_{n}}{n}$ & and & $x_{i}=x_{1}%
-\lambda_{i}, \, 2\leq i\leq n,$
\end{tabular}
\
\]
then the matrix in (\ref{mX}) realizes $\Lambda$.
\end{theorem}

\begin{lemma}
\cite{PP}
Let $\mathbf{e}$ and $J$ be the all ones column vector and the all ones square matrix, respectively.
\label{mM}
Let
$$M=\begin{pmatrix}
1&\mathbf{e}^{T}\\
\mathbf{e}&-I_{n-1}
\end{pmatrix}$$ 
then 
$$M^{-1}=\frac{1}{n}\begin{pmatrix}
1&\mathbf{e}^{T}\\
\mathbf{e} &J-nI_{n-1}
\end{pmatrix}$$
\end{lemma}
The following fact is an immediate consequence of Lemmas \ref{pappa copy(1)} and \ref{mM}.
\begin{remark}
\label{arbitrary_list}
Let $\Lambda=\left(\lambda_{1},\ldots, \lambda_{n}\right)$ and consider $\mathbf{x^{T}}=M^{-1}\Lambda^{T}$ then $\Lambda$  is the list of eigenvalues of the matrix $X$ in (\ref{mX}), where $\mathbf{x}=\left(x_{1},\ldots.x_{n}\right).$
\end{remark}

\noindent The following notions will be used in the sequel.

\begin{definition}
{\rm
\label{ept}
Let $\mathbf{\tau }=\left( \tau _{1},\ldots ,\tau _{n}\right) $ be an $n$-tuple whose elements are permutations in the symmetric group\ $S_{n}$, with $\tau _{1}=id$.\ Let $\mathbf{a=}\left( a_{1},\ldots ,a_{n}\right) \in
\mathbb{C}^{n}$. Define the row-vector,
\begin{equation*}
\tau _{j}\left( \mathbf{a}\right) =\left( a_{\tau _{j}\left( 1\right)
},\ldots ,a_{\tau _{j}\left( n\right) }\right)
\end{equation*}%
and consider the matrix
\begin{equation}
\tau \left( \mathbf{a}\right) =%
\begin{pmatrix}
\tau _{1}\left( \mathbf{a}\right)   &
\tau _{2}\left( \mathbf{a}\right)  &
\ldots &
\tau _{n-1}\left( \mathbf{a}\right)  &
\tau _{n}\left( \mathbf{a}\right)
\end{pmatrix}^{T}
.  \label{permut}
\end{equation}

\noindent A square matrix $A$, is called $\mathbf{\tau}$\emph{-permutative} if
$A=\tau \left( \mathbf{a}\right) $ for some $n$-tuple $\mathbf{a}$.}
\end{definition}

\begin{definition}\cite{MAR}
{\rm
If $A$ and $B$ are $\mathbf{\tau }$-permutative by a common vector
$\mathbf{\tau }=\left( \tau _{1},\ldots ,\tau _{n}\right)$
then they are called \textit{permutatively equivalent}.}
\end{definition}

\noindent We now need to define the following concept.

\begin{definition}
If $A$ and $B$ are permutatively equivalent matrices then $B$ is \textit{$A$-like permutative} or $A$ is \textit{$B$-like permutative}.
\end{definition}

\begin{remark}

\label{important2} A permutative matrix $A$ defines the class of permutatively equivalent matrices: The class of the $A$-like permutative matrices. Let $\sigma_{1}$ be an arbitrary Sule\u{\i}manova
spectra, then the corresponding realizing matrix $X_{\sigma_{1}}$ 
given by Theorem \ref{pappa} is a $X$-like permutative matrix, where $X$ is as the matrix in (\ref{mX}). Furthermore,  by Lemma \ref{pappa copy(1)} and Remark \ref{arbitrary_list} it is easy to check that given an arbitrary $n$-tuple (not necessarily into the NIEP) there exist a solution which is $X$ -like permutative, where $X$ is as the matrix in (\ref{mX}). By simplicity, during the paper we use $X$-like permutative matrix to refer a $X$-like permutative matrix, where $X$ is as the matrix in (\ref{mX}).
\end{remark}
\noindent The following definition concerns to the spectra.

\begin{definition}
\label{XlikeperSpc}
Given the list $\Lambda =\left( \lambda_0 ,\lambda _{1},\lambda
_{2},\ldots ,\lambda _{n-1}\right) $ we say that $%
\Lambda $ is \textit{ $X$-like permutative list} if there exists a nonnegative $X$-like permutative matrix whose spectrum is $\Lambda$. 
\end{definition}
 
 By Theorem \ref{pappa} the Sule\u{\i}manova lists are $X$-like permutative. Also the list $\Lambda =\left( 24 ,5,6,7,8\right) $ is a $X$-like permutative list as the $X$-like permutative matrix whose first row is $(10,2,3,4,5)$ realizes $\Lambda$.
 
 \begin{remark}
 \label{real-spectrum}
 By the equations in Theorem \ref{pappa}, a $X$-like permutative spectrum is always a real spectrum.
 \end{remark}
\noindent The following result is a necessary and sufficient condition for the list $\Lambda=\left(\lambda_1,\lambda_2,\ldots,\lambda_n\right)$ to be a $X$-like permutative list. 
\begin{theorem}
\label{necessary_suff}
Let $\Lambda=\left(\lambda_1,\lambda_2,\ldots,\lambda_n\right),$ be a real list, where $\lambda_1=\max _{ 2\leq \ell \leq n }{ \left| { \lambda  }_{ \ell  } \right|  } $, and $\sum _{ j=1  }^{ n }{ { \lambda  }_{ j }} \geq0.$ Then the set formed by the components of $\Lambda$ are the eigenvalues of a nonnegative irreducible $X$-like permutaitive matrix if and only if
\begin{eqnarray}
\label{ineq_perm}
{ \lambda  }_{ 1 }+(n-1){ \lambda  }_{ \nu (i) }\geq \sum _{ j=2,\  j\neq i  }^{ n }{ { \lambda  }_{ \nu (j) } },
\end{eqnarray}
for all $2\leq i\leq n$ and for some $\nu \in S_{n-1},$ the symmetric group of the permutations of the set $\{2,\ldots,n\}.$
In consequence, the expression $$\min _{ \nu \in { S }_{ n-1 } }{ \max _{ 2\leq i\leq n }{ \sum _{ j=2,\quad j\neq i }^{ n }{ { \lambda  }_{ \nu (j) } }  }  } -(n-1){ \lambda  }_{ \nu (i) }$$ is a Guo index for $X$-like permutative lists.
\end{theorem}

\begin{proof}
If there exits a permutation $\nu \in S_{n-1}$ such that the inequality in (\ref{ineq_perm}) holds then the vector 
$\mathbf{x^{T}}=M^{-1}\mathbf{\Lambda^{T}_{\nu}},$ 
is nonnegative, where $$\mathbf{\Lambda_{\nu}}=\left(\lambda_1,\lambda_{\nu(2)},\ldots,\lambda_{\nu(n)}\right),$$ by the matrix equations in Remark \ref{arbitrary_list} we conclude that a nonnegative $X$-like permutative matrix whose first row is $\mathbf{x}$ has the components of $\Lambda$ as eigenvalues. 
Reciprocally, if the components of $\Lambda$ are the eigenvalues of a nonnegative $X$-like permutative matrix whose first row is $\mathbf{x}$, by Remark \ref{arbitrary_list} $\mathbf{x}^{T}=M^{-1}\Lambda^{T}_{\mu} $ for some permutation $\mu$ in $S_{n-1}$, then by the non-negativity of $\mathbf{x}$ the inequalities in (\ref{ineq_perm}) hold for the  permutation $\mu$.
\end{proof}

\bigskip

\noindent By simplicity, during the paper, a $X $-like permutative matrix whose first row is $\mathbf{x}=\left(x_1,\ldots,x_n\right)$ is written 
\begin{eqnarray*}
per_X\left(x_1,\ldots,x_n\right) \quad \text{or} \quad per_X\left(\mathbf{x}\right).
\end{eqnarray*}
\bigskip
\noindent For $\mathbf{\tau}$-permutative or $\cdot \ $-like permutative matrices an analogous property of circulant matrices was proven in \cite{MAR}.

\begin{proposition}
\label{cristi}
Let $\left\{ A_{i}\right\} _{i=1}^{k}$ be a family of pairwise permutatively equivalent square matrices.
Let $\left\{ \gamma _{i}\right\} _{i=1}^{k}$ be a set of complex numbers. Consider
\begin{equation}
A=\sum\limits_{i=1}^{k}\gamma _{i}A_{i}.
\end{equation}
Then $A$ is a $A_{1}$-like permutative matrix.
\end{proposition}

\section{Matrices partitioned into circulant blocks}

\noindent The class of circulant matrices and their properties are introduced in \cite{Karner}. 
Let $\mathbf{a}=\left( a_{0},a_{1},\ldots ,a_{m-1}\right)$ be $m$-tuple of complex numbers.

\begin{definition} \cite{Karner}
A \emph{\ real circulant matrix} is a matrix of the form
\begin{equation*}
circ\left( \mathbf{a}\right) =
\begin{pmatrix}
a_{0} & a_{1} & \ldots  & \ldots & a_{m-1} \\
a_{m-1} & a_{0} & a_{1} & \ldots & a_{m-2} \\
a_{m-2} & \ddots  & \ddots  & \ddots  & \vdots  \\
\vdots  & \ddots  & \ddots  & a_{0} & a_{1} \\
a_{1} & \ldots  & a_{m-2} & a_{m-1} & a_{0}
\end{pmatrix}
\end{equation*}
\end{definition}

\bigskip

\noindent The next concepts can be seen in \cite{Karner}. The entries of the unitary discrete Fourier transform (DFT) matrix $F=\left(
f_{pq}\right) $ are given by
\begin{equation}
\label{fourier-m}
f_{pq}=\frac{1}{\sqrt{n}}\omega ^{pq},\quad 0\leq p, q\leq m-1\ ,
\end{equation}%
where
\begin{equation}
\label{omega_root}
\omega =\cos \frac{2\pi }{m}+i\sin \frac{2\pi }{m}.
\end{equation}%

\noindent The following results characterize the circulant spectra.

\begin{theorem}
\label{teorema 22}
\cite{Karner} Let $\mathbf{a}=(a_0,\ldots,a_{m-1})$ and $A(\mathbf{a})=circ(a_0,\ldots,a_{m-1}).$ Then $$A\left( \mathbf{a}\right) =FD \left( \mathbf{a}\right) F^{\ast },$$ with
$$
D\left( \mathbf{a}\right) =diag\left( \lambda _{0}\left( \mathbf{a}\right) ,\lambda
_{1}\left( \mathbf{a}\right) ,\ldots ,\lambda _{m-1}\left( \mathbf{a}\right) \right) $$
and
\begin{eqnarray}
\label{eigenvalues}
\text{\ }\lambda _{k}\left( \mathbf{a}\right) =\sum\limits_{\ell=0}^{m-1}a_{\ell}\omega
^{k\ell}\text{,\quad\ }0\leq k \leq m-1.
\end{eqnarray}
\end{theorem} 

\begin{corollary}
\label{fund} Let $\mathbf{a}$ defined as in Theorem \ref{teorema 22} and consider $$\Lambda=\Lambda(\mathbf{a})=\left( \lambda _{0}\left( \mathbf{a}\right) ,\lambda _{1}\left(
\mathbf{a}\right) ,\ldots ,\lambda _{m-1}\left( \mathbf{a}\right) \right) .$$
Then,
\begin{eqnarray}
\label{coefficients}
a_{k}=\frac{1}{m}\sum\limits_{\ell=0}^{m-1}\lambda _{\ell}\omega ^{-k\ell}\text{%
,\quad } 0 \leq k \leq m-1.
\end{eqnarray}
\end{corollary}

\noindent In \cite{AMRH}, using the techniques in \cite{MAR} for block matrices with circulant blocks, the following spectral result was proven.

\begin{theorem}
\cite{AMRH}
\label{main}
Let $K$ be an algebraically closed field of characteristic $0$ and suppose
that $A=\left( A(i,j)\right) $ is an $mn$-by-$mn$ matrix partitioned into $n^2$ circulant blocks of $m$-by-$m$ matrices, where for $ 1\leq i,j\leq n,$
\begin{equation}
\label{mtcsa}
A=\left( A(i,j)\right), \ A(i,j)=circ\left(\mathbf{a}(i,j)\right),
\end{equation}
where
\begin{eqnarray*}
\mathbf{a}(i,j)=(a_0(i,j),\ldots,a_{m-1}(i,j)),
\\
a_k(i,j)\in K,\ 1\leq i,j\leq n, \quad  0\leq k\leq m-1.
\end{eqnarray*}
Then
\begin{eqnarray}
\label{unionofsets}
\sigma \left( A\right) =\bigcup_{k=0}^{m-1}\sigma \left( S_k\right), 
\end{eqnarray}
where
\begin{align}
\label{mtcsk}
S_k&=\left( s_k(i,j)\right)_{1\leq i,j\leq n}, \
s_k(i,j)=\sum\limits_{\ell=0}^{m-1}a_{\ell}\left(i,j\right)\omega
^{k\ell}.
\end{align}
\end{theorem}

\noindent The next result is a direct consequence of Theorem \ref{main} and of (\ref{mtcsk}).
\begin{corollary}
\cite{AMRH}
\label{choose}
Consider $\mathbf{a}(u,v)=(a_0(u,v),\ldots, a_{m-1}(u,v))$ and $A(u,v)=circ\left(\mathbf{a}(u,v)\right).$ For $1\leq \ell \leq m$, let matrix $S_{\ell}$, related to the circulant blocks $A(u,v)$ defined in (\ref{mtcsk}). For $ 0\leq k \leq m-1,$ let 
\begin{eqnarray*}
L_k&=&\frac{1}{m}\sum\limits_{\ell=0}^{m-1} S _{\ell}\omega ^{-k\ell}, \qquad \text{then} \\
L_k
&=&\begin{pmatrix}
a_k(1,1) & a_k(1,2)&\ldots& a_k(1,m)\\
\vdots&\vdots&\ddots&\vdots\\
a_k(m,1) & a_k(m,2)&\ldots& a_k(m,m)
\end{pmatrix},
\end{eqnarray*}

In consequence, the matrix $A$ in (\ref{mtcsa}) is nonnegative if and only if for all $0\leq k \leq m-1$ the matrix $L_k$ is nonnegative and, in this case the matrix $S_{0}$ is nonnegative. 
\end{corollary}

\noindent From Theorem \ref{main} the following inverse result was obtained.

\begin{corollary}
\cite{AMRH}
\label{converse2}
Let $\left( S_{\ell }^{'}\right) _{\ell =0}^{m-1}$ be $m$ $n$-by-$n$ complex matrices. Then there exists a matrix $A^{'}$ partitioned into blocks where each block is circulant and whose spectrum is given by
\begin{eqnarray}
\label{unionofsets2}
\sigma \left( A^{'}\right) =\bigcup_{\ell=0}^{m-1}\sigma \left( S_{\ell}^{'}\right),
\end{eqnarray}
\end{corollary}

\noindent The following notions will be used in the sequel.

\begin{definition}
\cite{AMRH}
A matrix partitioned into blocks is called \textit{block permutative matrix} when all its row blocks (up to the first one) are permutations of precisely its first row block.
\end{definition}

\noindent The class of  the block $X$-like permutative  matrices is defined as follows.

\begin{definition}
Let $X=\left(x_{uv}\right)$ be the square matrix in (\ref{mX}). A block permutative matrix with circulant blocks $A=\left(circ(\mathbf{a}(u.v)\right)$ is a block $X$-like permutative matrix providing that, for all  $1\leq u,v\leq n$,
the block $circ\left(\mathbf{a}(u,v)\right)$ is the $(u,v)$-th block in $A$ if and only if $x_{uv}$ is the $(u,v)$-th entry in $X$.
\end{definition} 
\begin{theorem}
\cite{AMRH}
Let $A$ be the matrix partitioned into blocks as defined in (\ref{mtcsa}). For $0\leq k \leq m-1$ let $S_k$ be the class of matrices related to $A$ defined in (\ref{mtcsk}). The matrix $A$ is a block permutative matrix with circulant blocks if and only if the matrices $S_k$ are pairwise permutatively equivalent.
\end{theorem}

\section{An inverse problem related to block permutative matrices with circulant blocks}

\noindent In this section we study the Guo index to block permutative matrices with circulant blocks, which are a subclass of permutative matrices. Let $\mathbf{q}_1,\mathbf{q}_2,\ldots,\mathbf{q}_m$ denotes the canonical vectors

\noindent To our purpose we establish the following result. 
\begin{theorem}
\label{strongguo}
Let $E=(\varepsilon_{ij})$ be an $n$-by-$m$ matr\noindent ix where the multiset $\left\{ E\right\} $ formed by the entries of $E$ is closed under complex conjugation, the set of the entries in the first column of $E$ is a $X$-like permutative list and the sum of all row sums of $E$ is nonnegative. Suppose that $\varepsilon_{11}$ is positive and has the largest absolute value among the absolute values of entries of $E$.  Moreover, for $1\leq \ell \leq \left \lfloor \tfrac{m}{2} \right \rfloor,$
\begin{eqnarray}
\label{conjugatecondition3}
E\mathbf{q}_{\ell+1}=\overline{E\mathbf{q}_{m-\ell+1}}
\end{eqnarray}
that is, the entries of the $(m-\ell+1)$-th column of $E$ are the corresponding complex conjugate entries of the $(\ell+1)$-th column of $E.$ Note that for $m=2h$ the condition in (\ref{conjugatecondition3}) implies that the column $Eq_{h+1}$ of $E$ has real entries.  
If
\begin{eqnarray}
\label{new1}
\varepsilon_{11}\geq \Phi
\end{eqnarray}
with
\[
\Phi =\max_{\substack{ 0\leq k\leq m-1 \\ 0\leq j\leq n-1}}-\left[
\sum_{p=1}^{n-1}\varepsilon _{(p+1)1}+\sum_{\ell
=1}^{m-1}\sum_{p=0}^{n-1}\varepsilon _{(p+1)(\ell +1)}\omega ^{-k\ell
}-n\sum_{\ell =1}^{m-1}\varepsilon _{(p+1)(\ell +1)}\omega ^{-k\ell }\right]
\]            
Then $\left \{ E \right \}$ is the spectrum of a nonnegative $X$-like block permutative matrix $A$ whose blocks are circulant.
\end{theorem}

\begin{proof}

By the conditions of the statement there exists a nonnegative permutative  \[S_0:=per_X\left(
s_{00},s_{10},\ldots ,s_{(n-1)0}\right) \] whose spectrum is $\left \{ Eq_{1} \right \}$ (the set of the entries in $Eq_1$).
The condition in (\ref{conjugatecondition3}) implies that for $1\leq \ell\leq \left \lfloor \frac{m}{2} \right \rfloor$ the 
$X$-like permutative matrices $S_{\ell}$ and $S_{m-\ell}$ whose spectrum are $\left \{ E\mathbf{q}_{(\ell+1)} \right \}$ and $\left \{ E\mathbf{q}_{(m-\ell+1)} \right \}$, respectively, are related by
$\overline{S}_{\ell}=S_{m-\ell}$. For $ 1 \leq \ell \leq m-1$, 
suppose that
\begin{equation*}
S_{\ell}=per_X\left( s\left( \ell\right) \right) \text{, with }s\left( \ell\right)
=\left( s_{0\ell},s_{1\ell},\ldots ,s_{(n-1)\ell}\right) ,
\end{equation*}%
where
\begin{equation}
s\left( \ell\right)^{T} =M^{-1} E\mathbf{q}_{\ell+1} \label{circsl},
\end{equation}
then 

\begin{eqnarray}
\label{s-entries}
s_{0\ell}&=&\frac{1}{n}\left[\sum_{p=1}^{n}\varepsilon_{p(\ell+1)}\right] \quad and\\
s_{j\ell}&=&\frac{1}{n}\left[\sum_{p=1}^{n}\varepsilon_{p(\ell+1)}-n\varepsilon_{(j+1)(\ell+1)}\right], \quad 1\leq j\leq n-1.
\end{eqnarray}
For $0\leq k \leq m-1$, the entries of the $n$-by-$n$ $X$-like permutative matrix can be obtained, using equation (\ref{coefficients}) and the entries of the sums
\begin{equation}
L_k=\frac{1}{m}S_0+\frac{1}{m}\sum_{\ell =1}^{m-1}S_{\ell }\omega ^{-k\ell
}, \label{circlkp1}
\end{equation}

By Proposition \ref{cristi} the linear combination of $X$-like permutative matrices are  $X$-like permutative then matrices $L_k$ in (\ref{circlkp1})  are $X$-like permutative. Suppose that $$L_k=per_X\left(a_0(k),\ldots, a_{n-1}(k)\right).$$
From (\ref{circlkp1}), for $ 0\leq j \leq n-1,$ the following
holds

\begin{eqnarray*}
a_j(k) &=&\frac{1}{m}\left( s_{j0}+\sum_{\ell =1}^{m-1}\omega ^{-k\ell
}s_{j\ell } \right).
\end{eqnarray*}
Using (\ref{s-entries}), for $1\leq j \leq n-1$ we must have
\small{
\begin{eqnarray*}
0 \leq a_j(k)&=&\frac{1}{m}\left( s_{j0}+\sum_{\ell =1}^{m-1}\omega ^{-k\ell
}s_{j\ell } \right)\\
             & = &\small{\frac{1}{m}}\frac{1}{n}\left(
               \left[\varepsilon_{11}+\sum\limits_{p=1}^{n-1}\varepsilon_{(p+1)1}\right]+\sum_{\ell =1}^{m-1}\omega ^{-k\ell}\left[\sum\limits_{p=0}^{n-1}\varepsilon_{(p+1)(\ell+1)}-n\varepsilon_{(j+1)(\ell+1)}\right]
               \right) \\
             &= &\small{\frac{1}{mn}}\left(\varepsilon_{11}+\sum\limits_{p=1}^{n-1}\varepsilon_{(p+1)1}+\sum_{\ell =1}^{m-1}\sum\limits_{p=0}^{n-1}\varepsilon_{(p+1)(\ell+1)}\omega ^{-k\ell}-n\sum_{\ell =1}^{m-1}\omega ^{-k\ell}\varepsilon_{(j+1)(\ell+1)}\right)
\end{eqnarray*}}
\noindent for all $0\leq k\leq  m-1.$ Therefore, the last condition implies the inequality in (\ref{new1}).
\end{proof}

\begin{remark}
To computational aims it is worth note that the columns of the product $$L=\frac{1}{\sqrt{m}}M^{-1}EF^{*}$$ are the first rows of the matrices $L_0,\ldots L_{m-1}.$
\end{remark}

\noindent Now, one can formulate the following question.
\begin{problem}
Which condition (or conditions) is (or are) necessary and sufficient for the existence of a nonnegative matrix, $A$ $X$-like block permutative matrix with circulant blocks whose spectrum equals to $\left \{ E \right \}$,? where $E=\left( \varepsilon _{ij}\right)$ is as in Theorem \ref{strongguo},
\end{problem}

 \noindent Let us consider the set
\[
\boldsymbol{P}=\left\{ f:\left\{ E\right\} \rightarrow \left\{ E\right\} :f%
\text{ }is\ bijective\right\}
\]%
\begin{definition}
The function $f\in \boldsymbol{P}$ is said to be \textit{$E$-nonnegative spectrally stable ($E$-NNSS)} if
the matrix $E\left( f\right) =\left( f\left( \varepsilon _{ij}\right)
\right) $ is such that
$f\left( \varepsilon _{11}\right) $ has the maximum absolute value among $%
f\left( \varepsilon _{ij}\right), $ the first column of $E\left( f\right) $ is real and it is a $X$-like permutative list, the
$\left( m-\ell +1 \right) $-th column of $E\left( f\right) $ is the complex
conjugate column of the $\left( \ell +1\right) $-th column  of  $E\left(
f\right) $, and then
\[
E\left( f\right) \mathbf{q}_{(m-\ell +1) }=\overline{E\left( f\right) }\mathbf{q}_{\left( \ell
+1\right) },
\]%
for $1 \leq \ell \leq \left \lfloor \frac{m}{2}  \right \rfloor.$
\end{definition}

For instance:

\begin{enumerate}

\item The identity function of $E$ into $E$ is clearly $E$-NNSS.
\item   $f:\left\{ E\right\} \rightarrow \left\{ E\right\} $ defined by
\[
f\left( \varepsilon _{ij}\right) =\left \{
\begin{array}{ll}
\varepsilon _{ij} & j\neq 2 \quad {\small \text{and}} \quad j\neq m, \\
\varepsilon_{i m } & j=2, \\
\varepsilon _{i2} & j=m;
\end{array}
\right.
\]
is $E$-NNSS.
\item  $f:\left\{ E\right\} \rightarrow \left\{ E\right\} $ defined by
\[
f\left( \varepsilon _{ij}\right) =\overline{\varepsilon }_{ij}\mbox{\, for all\,} i,j
\]
is $E$-NNSS.
\item  $f:\left\{ E\right\} \rightarrow \left\{ E\right\} $ defined by
\[
f\left( \varepsilon_{ij}\right) =
\left\{
\begin{array}{ll}
\varepsilon_{ij} & \mbox{\, if \,} j=1,\\
\overline{\varepsilon}_{ij} & \mbox{\,if \,} j\neq 1;
\end{array}
\right.
\]
is $E$-NNSS.
\end{enumerate}
We denote by $\boldsymbol{P}^{*}$ the subset of $\boldsymbol{P}$ formed by all $E$-NNSS bijections of $E$.
Note that if $f \in P^{*}$ then $f(\varepsilon_{11})=\varepsilon_{11}.$

\begin{theorem}
Let $E=(\varepsilon_{ij})$ be an $n$-by-$m$ matrix such that its first row is a $X$-like permutative spectrum. Moreover, the multiset $\left\{E\right\}$ formed by the entries of $E$ satisfies the conditions in Theorem \ref{strongguo}. 
The multiset $\left \{ E \right \}$ is the spectrum of a block $X$-like permutative matrix $A$ whose blocks are circulant if and only if

\begin{eqnarray}
\label{maria}
\varepsilon_{11} &\geq &
\min_{f\in \boldsymbol{P}^{*}}
\max\left\{\Theta: 0\leq k \leq m-1, \ 0\leq j \leq n-1\right\}
\end{eqnarray}
where,
\begin{eqnarray*}
\Theta &  = &
 -\left[\sum\limits_{p=1}^{n-1}f(\varepsilon_{(p+1)1})+\sum_{\ell =1}^{m-1}\sum\limits_{p=0}^{n-1}f(\varepsilon_{(p+1)(\ell+1)})\omega ^{-k\ell}-n\sum_{\ell =1}^{m-1}\omega ^{-k\ell}f(\varepsilon_{(j+1)(\ell+1)}) \right].
\end{eqnarray*}

\end{theorem}

\begin{proof}
Following the same steps of the above proof this time replacing $\varepsilon_{ij}$ by $f(\varepsilon_{ij})$ we arrive at the inequality in (\ref{new1}). After taking the minimum when the function $f$ vary into $\boldsymbol{P}^{*}$, the inequality (\ref{maria}) is obtained.
\end{proof}

\begin{example}
    Let consider
    \[
    E=%
    \begin{pmatrix}
    23.9 & i & -i \\
    -3 & 1-7i & 1+7i\\
    0 & -3+i & -3-i%
    \end{pmatrix}%
    \]%
   \noindent  The nonnegative $X$-like permutative matrix $per_X(6.9667, 9.9667,6.9667) $ has the set formed by the components of $E\mathbf{q}_1$ as spectrum. In addition we have
    \begin{eqnarray*}
    L_{0} &=&per_X\left(1.8778, 2.2111, 3.8778\right)\\
    L_{1} &=&per_X\left(1.5822, 6.9570, 0,0048\right)\\
    L_{2} &=&per_X\left(3.5067, 0.7986, 3.0840\right).\\
    \end{eqnarray*}%
    
   \noindent In consequence, the circulant blocks are
    
    \begin{eqnarray*}
    A_{0} &=&circ\left(1.8778, 1.5822, 3.5067\right)\\
    A_{1} &=&circ\left(2.2111, 6.9570, 0.7986\right)\\
    A_{2} &=&circ\left(3.8778, 0,0048, 3.0840\right)\\
    \end{eqnarray*}%
    
  \noindent  Therefore, the $X$-like block permutative matrix
    \[\begin{pmatrix}
    A_0 & A_1 & A_2\\
    A_1 & A_0 & A_2\\
    A_2 & A_1 & A_0
    \end{pmatrix}\]
    
 \noindent has spectrum the multiset $\left\{E\right\}.$
    \end{example}
    
    \begin{example}
    Let consider
    \[
    E_{1}=%
    \begin{pmatrix}
    5 & -3 & -2 & -3\\
    2 & 1 & 2 & 1%
    \end{pmatrix}%
    \]%
    \noindent  The nonnegative $X$-like permutative matrix $per_X(3.5,1.5) $ has the set formed by the components of $E\mathbf{q}_1$ as spectrum. In addition we have
   \noindent the matrices
    \begin{eqnarray*}
    L_{0} =per_{X}\left(1.5,0\right), \ L_{1} =per_{X}\left(2,2\right), \ L_{2}= per_{X}\left(2.5, 2\right), \ L_{3}= per_{X}\left(2, 2\right).
    \end{eqnarray*}
    
   \noindent   In consequence, we obtain
    \begin{eqnarray*}
    A_{0} =circ\left(1.5,2, 2.5, 2\right), \ A_{1} =circ\left(0,2,2,2\right).
    \end{eqnarray*}
  
 \noindent   Therefore, the $X$-like block permutative matrix
    \[\begin{pmatrix}

    A_0 & A_1 \\
    A_1 & A_0 
    \end{pmatrix}\]
    
\noindent    has spectrum the multiset $\left\{E_1\right\}.$
    
\end{example}

\begin{example}
\noindent    Let consider
    \[
    E_{2}=%
    \begin{pmatrix}
    2.5 & 0.25i & 0 & -0.25i \\
    -1 & 0.5-i & 0 & 0.5+i%
    \end{pmatrix}%
    \]%
  \noindent  thus, it is obtained
    \begin{eqnarray*}
    L_{0}=per_X\left(0.25, 0.375\right),\ L_{1}= per_X\left(0, 0.75\right),\\& \\  L_{2}= per_X\left(0.125, 0.5\right)\ L_{3}= per_X\left(0.375, 0.125\right).
    \end{eqnarray*}
\noindent    In consequence,
    \begin{eqnarray*}
    A_{0} = circ\left(0.25,0.0, 0.125,0.375 \right),\
    A_{1} = circ\left(0.375, 0.75, 0.5, 0.125\right).
    \end{eqnarray*}

\noindent    Therefore, the $X$-like block permutative matrix
    \[\begin{pmatrix}
    A_0 & A_1 \\
    A_1 & A_0 
    \end{pmatrix}\]
    
  \noindent  has spectrum the multiset $\left\{E_{2}\right\}.$
    
\noindent By consider in $E_2(1,1)=2.49$ in a place of $E_2(1,1)=2.5$ it is obtained
\begin{eqnarray*}
    L_{0} &=&per_X\left(0.2488,0.3738\right)\\
    L_{1} &=&per_X\left(-0.0012, 0.7488\right)\\
    L_{2}&=& per_X\left(0.1238, 0.4987\right)\\
    L_{3}&=& per_X\left(0.3738, 0.1238\right).
    \end{eqnarray*}
    \noindent    Therefore, the $X$-like block permutative matrix that we can construct is no nonnegative.
\end{example}

\begin{example}
\noindent    Let consider
    \[
    E_{3}=%
    \begin{pmatrix}
    4 & 1 & 1 \\
    -1 & -2.5 & -2.5%
    \end{pmatrix}%
    \]%
  \noindent Thus are obtained the matrices
    \begin{eqnarray*}
    L_{0} =per_X\left(0,2\right),\   L_{1} =per_X\left(0.75,0.25\right),\
    L_{2}= per_X\left(0.75, 0.25\right).
    \end{eqnarray*}
\noindent    In consequence, 
    \begin{eqnarray*}
    A_{0} =circ\left(0,0.75, 0.75\right),\    A_{1} =circ\left(2,0.25,0.25\right).
    \end{eqnarray*}

\noindent    Therefore, the $X$-like block permutative matrix
    \[\begin{pmatrix}
    A_0 & A_1 \\
    A_1 & A_0 
    \end{pmatrix}\]
    
  \noindent  has spectrum the multiset $\left\{E_{3}\right\}.$
    Considering the multiset formed by the entries of $E_{1}, 4$ is the least Perron root that can be considered because the trace becomes negative if the spectral radius is diminished.
\end{example}

\textbf{References}.


\begin{thebibliography}{120}

\bibitem{AMRH} E. Andrade, C. Manzaneda, H. Nina, M. Robbiano. Block matrices and Guo's Index for block circulant matrices with circulant blocks. Paper Accepted.


\bibitem{Berman} A. Berman, R. J. Plemmons, Nonnegative matrices in the Mathematical Sciences, SIAM Publications, Philadelphia, 1994.

\bibitem{Boro} A. Borobia. On nonnegative eigenvalue problem, Lin. Algebra
Appl. 223/224 (1995): 131-140, Special Issue honoring Miroslav Fiedler and
Vlastimil Pt\'{a}k.


\bibitem{Friedland} S. Friedland, On an inverse problem for nonnegative and
eventually nonnegative matrices, Israel T. Math. 1, 29 (1978): 43-60.

\bibitem{Fiedler} M. Fiedler, Eigenvalues of nonnegative symmetric
matrices, Lin. Algebra Appl. 9 (1974): 119-142.


\bibitem{GUO} W. Guo, Eigenvalues of nonnegative matrix, Linear Algebra and its Applications 266 (1997): 261--270.


\bibitem{Johnson} C. R. Johnson, Row stochastic matrices similar to doubly
stochastic matrices, Lin. and Multilin. Algebra 2 (1981): 113-130.

\bibitem{Laffey} C. Johnson, T. Laffey, R. Loewy, The real and symmetric
nonnegative inverse eigenvalue problems are different, Proc. Amer. Math
Soc.,12, 124 (1996): 3647-3651.



\bibitem{Karner} H. Karner, J. Schneid, C. W. Ueberhuber, Spectral decomposition of real circulant matrices. Lin. Algebra Appl. 367 (2003): 301-311.

\bibitem{Laffey1}T. Laffey, Extreme nonnegative matrices,\ Lin. Algebra
Appl. 275/276 (1998): 349-357. Proceedings of the sixth conference of the
international Linear Algebra Society (Chemnitz, 1996).

\bibitem{Laffey2} T. Laffey, Realizing matrices in the nonnegative inverse
eigenvalue problem, Matrices and group representations (Coimbra, 1998),
Textos Mat. S\'{e}r. B, 19, Univ. Coimbra, Coimbra, (1999): 21-31.

\bibitem{Laffey-Smigoc} T. Laffey, H. \v{S}migoc, Nonnegative realization of spectra having negative real parts, Linear Algebra Appl., 384 (2004): 199--206.

\bibitem{LMc} R. Loewy,  J. J. Mc Donald, The symmetric nonnegative inverse
eigenvalue problem for $5\times 5$ matrices, Linear Algebra Appl. 393
(2004): 275-298.

\bibitem{LwyLdn} R. Loewy, D. London, A note on an inverse problem for
nonnegative matrices, Lin. and Multilin. Algebra 6, 1 (1978/79): 83-90.


\bibitem{MAR} C. Manzaneda, E. Andrade, M. Robbiano, Realizable lists via the spectra of structured matrices, Lin. Algebra Appl. 534 (2017): 51-72.



\bibitem{PP} P. Paparella, Realizing Sule\u{\i}manova-type spectra
via permutative matrices, Electron. J. Linear Algebra, 31 (2016): 306-312.









\bibitem{Smgoc1} H. \v{S}migoc, The inverse eigenvalue problem for
nonnegative matrices, Linear Algebra Appl. 393 (2004): 365-374.

\bibitem{Smigoc2} H. \v{S}migoc, Construction of nonnegative matrices and the inverse
eigenvalue problem, Lin. and Multilin. Algebra 53, 2 (2005): 85-96.




\end{thebibliography}
\end{document}